\documentclass[reqno]{amsart}

\usepackage{amsmath,amssymb,amsthm,amsfonts,enumerate, enumitem,hyperref,cleveref}
\usepackage{dsfont}
\usepackage[all]{xy}
\usepackage[T1]{fontenc}
\usepackage{tikz}
\usepackage{pgfplots}
\pgfplotsset{compat=1.15}
\usepackage{mathrsfs}
\usetikzlibrary{arrows}

\DeclareMathOperator{\inn}{Inn}
\DeclareMathOperator{\laut}{LAut}
\DeclareMathOperator{\Span}{Span}

\DeclareMathOperator{\Max}{Max}
\DeclareMathOperator{\Min}{Min}

\theoremstyle{plain}
\newtheorem{theorem}{Theorem}[section]
\newtheorem{corollary}[theorem]{Corollary}
\newtheorem{proposition}[theorem]{Proposition}
\newtheorem{lemma}[theorem]{Lemma}

\theoremstyle{definition}

\newtheorem{remark}[theorem]{Remark}

\crefrangeformat{equation}{#3(#1)#4--#5(#2)#6}

\crefname{theorem}{Theorem}{Theorems}
\crefname{lemma}{Lemma}{Lemmas}
\crefname{corollary}{Corollary}{Corollaries}
\crefname{proposition}{Proposition}{Propositions}
\crefname{definition}{Definition}{Definitions}
\crefname{example}{Example}{Examples}
\crefname{remark}{Remark}{Remarks}
\crefname{conjecture}{Conjecture}{Conjectures}
\crefname{section}{Section}{Sections}
\crefname{equation}{\unskip}{\unskip}
\crefname{enumi}{\unskip}{\unskip}
\crefname{subsection}{Subsection}{Subsections}

\newcommand{\0}{\theta}
\newcommand{\ve}{\varepsilon}

\newcommand{\af}{\alpha}
\newcommand{\bt}{\beta}
\newcommand{\lb}{\lambda}

\newcommand{\gm}{\gamma}
\newcommand{\vf}{\varphi}

\newcommand{\dl}{\delta}

\newcommand{\sg}{\sigma}

\newcommand{\B}{\mathcal{B}}

\newcommand{\m}{{}^{-1}}
\newcommand{\sst}{\subseteq}

\newcommand{\wtl}{\widetilde}

\newcommand{\id}{\mathrm{id}}

\begin{document}
	\title[Regular Hom-Lie structures on incidence algebras]{Regular Hom-Lie structures\\ on incidence algebras}	
	\author{\'Erica Z. Fornaroli}
	\address{Departamento de Matem\'atica, Universidade Estadual de Maring\'a, Maring\'a, PR, CEP: 87020--900, Brazil}
	\email{ezancanella@uem.br}
	
	\author{Mykola Khrypchenko}
	\address{Departamento de Matem\'atica, Universidade Federal de Santa Catarina,  Campus Reitor Jo\~ao David Ferreira Lima, Florian\'opolis, SC, CEP: 88040--900, Brazil \and CMUP, Departamento de Matemática, Faculdade de Ciências, Universidade do Porto,
		Rua do Campo Alegre s/n, 4169--007 Porto, Portugal}
	\email{nskhripchenko@gmail.com}
	
	\author{Ednei A. Santulo Jr.}
	\address{Departamento de Matem\'atica, Universidade Estadual de Maring\'a, Maring\'a, PR, CEP: 87020--900, Brazil}
	\email{easjunior@uem.br}
	
	
	\begin{abstract}
		We fully characterize regular Hom-Lie structures on the incidence algebra $I(X,K)$ of a finite connected poset $X$ over a field $K$. We prove that such a structure is the sum of a central-valued linear map annihilating the Jacobson radical of $I(X,K)$ with the composition of certain inner and multiplicative automorphisms of $I(X,K)$. 
	\end{abstract}
	
	\maketitle
	
	\tableofcontents
	
	\section*{Introduction}
	
	A \textit{Hom-Lie algebra}~\cite{Makhlouf-Silvestrov08} over a field $K$ is a triple $(L,[\cdot,\cdot],\vf)$, where $L$ is a $K$-vector space, $[\cdot,\cdot]$ is an anti-commutative bilinear multiplication on $L$ and $\vf:L\to L$ is a linear map satisfying the so-called \textit{Hom-Jacobi identity}
	\begin{align}\label{hom-jacobi}
		[[a,b],\vf(c)]+[[b,c],\vf(a)]+[[c,a],\vf(b)]=0 
	\end{align}
	for all $a,b,c\in L$. The notion of a Hom-Lie algebra first appeared in~\cite{Hartwig-Larsson-Silvestrov06} (under a slightly different definition), where the authors gave a construction of such an algebra based on a commutative associative algebra with a $\sg$-derivation. A Hom-Lie algebra $(L,[\cdot,\cdot],\vf)$, in which $\vf$ is a homomorphism (resp. isomorphism) of $L$, is called \textit{multiplicative} (resp.~\textit{regular})~\cite{Makhlouf-Zusmanovich18}. If $(L,[\cdot,\cdot])$ is itself a (usual) Lie algebra, then by a (multiplicative, regular) \textit{Hom-Lie structure} on $L$ we mean a linear map $\vf:L\to L$ making $(L,[\cdot,\cdot],\vf)$ a (multiplicative, regular) Hom-Lie algebra. For example, $\vf=\lb\cdot\id$, where $\lb\in K$, is always a Hom-Lie structure called \textit{trivial}. It is multiplicative only when $\lb\in\{0,1\}$.
	
	Jin and Li~\cite{Jin-Li08} proved that all multiplicative Hom-Lie structures on a finite-dimensional simple Lie algebra are trivial. Xie, Jin and Liu improved the previous result and described arbitrary (not necessarily multiplicative) Hom-Lie structures on finite-dimensional simple Lie algebras in~\cite{Xie-Jin-Liu15}. The superalgebra case was investigated in \cite{Cao-Luo13,Yuan-Liu15}. Xie and Liu characterized Hom-Lie structures on the Virasoro algebra and on the loop and Cartan Lie algebras, completing thus a description of Hom-Lie structures on graded simple Lie algebras of finite growth~\cite{Xie-Liu17}. Makhlouf and Zusmanovich~\cite{Makhlouf-Zusmanovich18} showed that the space of Hom-Lie structures on an affine Kac–Moody algebra is linearly spanned by central Hom-Lie structures and the identity map. Benayadi and Makhlouf~\cite{Benayadi-Makhlouf14} established a correspondence between the class of involutive quadratic Hom–Lie algebras and the class of quadratic simple Lie algebras admitting involutive automorphisms. Chen and Yu~\cite{Chen-Yu20} proved that any regular Hom-Lie structure on a Borel subalgebra of a finite-dimensional simple Lie algebra over an algebraically closed field of characteristic zero is an extremal inner automorphism. Another paper by Chen and Yu~\cite{Chen-Yu22} is of special interest to us. In~\cite{Chen-Yu22} the authors describe regular Hom–Lie structures on the Lie algebra of \textit{strictly} upper triangular matrices over a field. Recall that the algebra $T_n(K)$ of \textit{all} upper triangular matrices of order $n$ over $K$ is a particular case of the incidence algebra $I(X,K)$ of a finite connected poset $X$. Lie automorphisms of $I(X,K)$ were fully described in~\cite{FKS}, so it is natural to use this description to characterize those Lie automorphisms which are (regular) Hom-Lie structures on $I(X,K)$ (endowed with the commutator Lie bracket $[\cdot,\cdot]$). This is done in the present paper.
	
	In \cref{sec-prelim} we give the necessary background on Lie automorphisms~\cite{FKS,FKS2} and usual automorphisms~\cite{Baclawski72,St} of $I(X,K)$. In \cref{0=identity} we obtain the conditions that a Lie automorphism $\vf$ of $I(X,K)$ must satisfy to be a regular Hom-Lie structure. These conditions turn out to be sufficient as well, whenever $\vf$ is elementary, as shown in \cref{0=id-and=sg=1=>vf-Hom-Lie}. As a consequence, a characterization of the \textit{elementary} Lie automorphisms of $I(X,K)$ that are regular Hom-Lie structures is summarized in \cref{vf-Hom_lie<=>0=id-and-sg=1}. After that, we are left to deal with a class of inner automorphisms of $I(X,K)$, which is done in \cref{xi-Hom-Lie<=>rho-in-Z}. The next \cref{vf-Hom-Lie<=>wtl-vf-and-xi-Hom-Lie} permits us to combine the results of \cref{0=id-and=sg=1=>vf-Hom-Lie,xi-Hom-Lie<=>rho-in-Z} into \cref{main-theorem}, which describes \textit{all} regular Hom-Lie structures on $I(X,K)$ in terms of simpler maps coming from~\cite{FKS}. A more explicit description is given in \cref{vf=xi_bt-circ-M_sg+nu}, specified to the case of posets of length one in \cref{dropping-bt_D=id}. As an application, we characterize regular Hom-Lie structures on the algebra of upper triangular matrices (seen as an incidence algebra) in \cref{Hom-Lie-T_n(K)}. We end the work with \cref{Hom-Lie-normal-subgroup} which is an analog of~\cite[Corollary 3.5]{Chen-Yu22} proving that the set of regular Hom-Lie structures on $I(X,K)$ is a normal subgroup of the group of all Lie automorphisms of $I(X,K)$.
	
	\section{Preliminaries}\label{sec-prelim}
	
	\subsection{Posets and incidence algebras}
	
	Let $(X,\le)$ be a finite poset. A \textit{chain} in $X$ is a linearly ordered subposet of $X$. The \textit{length} of a chain $C\sst X$ is defined to be $|C|-1$. The \textit{length} of $X$, denoted by $l(X)$, is the maximum length of all chains $C\sst X$. For simplicity, we write $l(x,y)$ for $l(\{z\in X : x\leq z\leq y\})$. The poset $X$ is \emph{connected} if for any pair of $x,y\in X$ there is a sequence $x=x_0,\dots,x_m=y$ in $X$ such that for all $0\le i\le m-1$ either $x_i\le x_{i+1}$ or $x_i\ge x_{i+1}$. We will denote by $\Min(X)$ (resp.~$\Max(X)$) the set of minimal (resp.~maximal) elements of $X$ and by $X^2_<$ (resp.~$X^2_{\leq}$) the set of pairs $(x,y)\in X^2$ such that $x<y$ (resp.~$x\leq y$).
	
	Let $X$ be a finite poset and $K$ a field. The \emph{incidence algebra} $I(X,K)$ of $X$ over $K$ (see~\cite{Rota64}) is the $K$-space with basis $\{e_{xy} : x\leq y\}$
	and multiplication given by
	$$
	e_{xy}e_{uv}=\dl_{yu}e_{xv},
	$$
	for any $x\leq y$, $u\leq v$ in $X$ (here $\dl_{yu}$ is the Kronecker delta). Given $f\in I(X,K)$, we write $f=\sum_{x\le y}f(x,y)e_{xy}$, where $f(x,y)\in K$. Let us denote $e_x := e_{xx}$. Then $I(X,K)$ is an associative $K$-algebra with identity element $\dl=\sum_{x\in X}e_x$.
	
	Throughout the rest of the paper $X$ will be a connected finite poset with $|X|>1$. Then $I(X,K)$ is a central algebra, by \cite[Corollary~1.3.15]{SpDo}.
	
	By \cite[Theorem~1.2.3]{SpDo}, the group of units of $I(X,K)$, denoted by $I(X,K)^{\ast}$, consists of all $f\in I(X,K)$ such that $f(x,x)\neq 0$ for all $x\in X$. It is well-known (see~\cite[Theorem~4.2.5]{SpDo}) that the Jacobson radical $J(I(X,K))$ of $I(X,K)$ is spanned by $\B:=\{e_{xy} : x<y\}$. \emph{Diagonal elements} of $I(X,K)$ are those $f\in I(X,K)$ satisfying $f(x,y)=0$ for $x\neq y$. They form a commutative subalgebra $D(X,K)$ of $I(X,K)$ with basis $\{e_{x} : x \in X\}$. Hence, each $f\in I(X,K)$ can be uniquely written as $f=f_D+f_J$ with $f_D\in D(X,K)$ and $f_J\in J(I(X,K))$.
	
	For each positive integer $m$, $J(I(X,K))^m=\Span_K\{e_{xy} : l(x,y)\geq m\}$, by \cite[Proposition~2.4]{FKS}. We will use the notation $J(I(X,K))^{0}=I(X,K)$. By \cite[Proposition 2.5]{FKS}, the center $Z(J(I(X,K)))$ of $J(I(X,K))$ is $\Span_K\{e_{xy}\in \B : x\in \Min(X) \text{ and } y\in\Max(X)\}$. It also coincides with the (bilateral) annihilator of $J(I(X,K))$.
	
	\subsection{Automorphisms and Lie automorphisms of $I(X,K)$}
	
	We are going to deal with two classes of automorphisms of $I(X,K)$, whose definitions are given below.
	
	An element $\sg\in I(X,K)$ is \emph{multiplicative} if $\sg(x,y)\neq 0$, for all $x\le y$, and $\sg(x,z)=\sg(x,y)\sg(y,z)$ whenever $x\le y\le z$. Every such $\sg$ determines the corresponding \emph{multiplicative automorphism} $M_{\sg}$ of $I(X,K)$ by $M_{\sg}(f)=\sg\ast f$, for all $f\in I(X,K)$, where $\sg\ast f$ is the Hadamard product given by $(\sg\ast f)(x,y)=\sg(x,y)f(x,y)$ for all $x,y\in X$. If $\sg, \tau \in I(X,K)$ are multiplicative, then so is $\sg\ast\tau$ and $M_\sg\circ M_\tau=M_{\sg*\tau}$ (see~\cite[7.3]{SpDo}).
	
	For each $\beta\in I(X,K)^{\ast}$ we denote by $\xi_{\beta}$ the inner automorphism of $I(X,K)$ given by $\xi_{\beta}(f)=\bt f \bt\m$ for all $f\in I(X,K)$. 
	
	Regarding the Lie automorphisms of $I(X,K)$, i.e. linear bijections of $I(X,K)$ preserving the commutator product $[f,g]=fg-gf$, we are going to use definitions and results from \cite{FKS}. 
	
	Let $C: u_1<u_2<\dots<u_m$ be a maximal chain in $X$ and $\0:\B\to \B$ a bijection. Then $\0$ is called \textit{increasing (resp.~decreasing) on $C$} if there exists a maximal chain $D: v_1<v_2<\dots<v_m$ in $X$ such that $\0(e_{u_iu_j})=e_{v_iv_j}$ for all $1\le i<j\le m$ (resp.~$\0(e_{u_iu_j})=e_{v_{m-j+1}v_{m-i+1}}$ for all $1\le i<j\le m$). Moreover, we say that $\0$ \textit{is monotone on maximal chains in $X$} if, for any maximal chain $C$ in $X$, $\0$ is increasing or decreasing on $C$. 
	
	Let $\0:\B\to \B$ be a bijection. A map $\sg:X^2_<\to K^*$ is \textit{compatible} with $\0$ if, for all $x<y<z$,
	\begin{align*}
		\sg(x,z)=
		\begin{cases}
			\sg(x,y)\sg(y,z), &\text{if }\0(e_{xz})=\0(e_{xy})\0(e_{yz}),\\
			-\sg(x,y)\sg(y,z), &\text{if }\0(e_{xz})=\0(e_{yz})\0(e_{xy}).
		\end{cases}
	\end{align*}
	In particular, if $\0$ is increasing on \textit{all} the maximal chains in $X$, then any such $\sg$ extends to a multiplicative element of $I(X,K)$ by setting $\sg(x,x)=1$ for all $x\in X$.
	There is also the notion of an \textit{admissible} bijection $\B\to \B$ whose definition is quite technical and will not be used explicitly in the present work (just let us mention that $\id_\B$ is always admissible). We refer the interested reader to~\cite[Definition 5.11]{FKS}. 
	
	Let $X=\{x_1,\dots, x_n\}$. Given an admissible monotone bijection $\0: \B\to \B$, a map $\sg:X_<^2\to K^*$ compatible with $\0$ and a sequence $c=(c_1,\dots,c_n)\in K^n$ such that $\sum_{i=1}^nc_i\in K^*$, we define in \cite[Definition~5.17]{FKS} the following Lie automorphism $\tau=\tau_{\0,\sg,c}$ of $I(X,K)$ where, for any $e_{xy}\in \B$, 
	$$
	\tau(e_{xy})=\sigma(x,y)\0(e_{xy})
	$$
	and $\tau|_{D(X,K)}$ is uniquely determined by
	$$
	\tau(e_{x_i})(x_1,x_1)=c_i,
	$$
	$i=1,\dots,n$, as in Lemmas~5.8 and 5.16 from \cite{FKS} (the admissibility of $\0$ guarantees that $\tau|_{D(X,K)}$ is well-defined). The Lie automorphisms of the form $\tau_{\0,\sg,c}$ are exactly the so-called \textit{elementary Lie automorphisms}. Moreover, the triple $(\0,\sg,c)$ is uniquely determined by $\tau$ (see~\cite[Definitions 4.1, 4.7 and Theorem 5.18]{FKS}).
	
	As in \cite{FKS}, we denote by $\laut(I(X,K))$ the group of Lie automorphisms of $I(X,K)$ and by $\wtl\laut(I(X,K))$ its subgroup of elementary Lie automorphisms. We will also use the notation $\inn_1(I(X,K))$ for the subgroup of inner automorphisms $\xi_{\beta}$ with $\bt_D=\dl$. Observe that such $\bt$ is uniquely determined by $\xi_\bt$, so in fact $\inn_1(I(X,K))$ is isomorphic to the subgroup of all $\bt\in I(X,K)^*$ with $\bt_D=\dl$.
	
	\begin{theorem}\cite[Theorem~4.15]{FKS}\label{semidireto}
		The group $\laut(I(X,K))$ is isomorphic to the semidirect product $\inn_1(I(X,K))\rtimes\wtl\laut(I(X,K))$.
	\end{theorem}
	
	Let $\varphi\in\laut(I(X,K))$. By \cref{semidireto}, $\varphi=\xi_\bt\circ \tau_{\0,\sg,c}$ for a unique quadruple $(\bt,\0,\sg,c)$ with $\bt_D=\dl$. In this case we write $\wtl\vf:=\tau_{\0,\sg,c}$, $\0_{\vf}:=\0$ and $\sg_{\vf}:=\sg$. For any $e_{xy}\in J(I(X,K))^i-J(I(X,K))^{i+1}$ we have
	\begin{align*}
		\vf(e_{xy})=\wtl\vf(e_{xy})+\rho_{xy},
	\end{align*}
	with $\wtl\vf(e_{xy})\in J(I(X,K))^i-J(I(X,K))^{i+1}$ and $\rho_{xy}\in J(I(X,K))^{i+1}$ (see \cite[Definition~4.1]{FKS}).

	\section{Regular Hom-Lie structures on $I(X,K)$}\label{sec-reg-Hom-Lie}

	\begin{remark}\label{psi-vf-psi-inv-Hom-Lie}
		Let $(L,[\cdot,\cdot])$ be a Lie algebra.
		\begin{enumerate}
			\item A linear map $\vf: L\to L$ is a Hom-Lie structure if and only if \cref{hom-jacobi} holds for all $a,b,c$ from a basis of $L$.
			\item If $\vf$ is a regular Hom-Lie structure on $L$ and $\psi$ is an arbitrary automorphism of $L$, then $\psi\circ\vf\circ\psi\m$ is also a regular Hom-Lie structure on $L$. Indeed, this follows by taking $a=\psi\m(a')$, $b=\psi\m(b')$ and $c=\psi\m(c')$ in \cref{hom-jacobi} and applying $\psi\m$ to both sides of \cref{hom-jacobi}. 
		\end{enumerate}
	\end{remark}
	
	\begin{lemma}\label{basic-formulas}
		Let $\vf\in\laut(I(X,K))$ be a regular Hom-Lie structure on $I(X,K)$. If $x<y<z$, then $[\vf(e_{xy}),e_{yz}]=[e_{xz},\vf(e_z)]$ and $[e_{xy},\vf(e_{yz})]=[\vf(e_x),e_{xz}]$.
	\end{lemma}
	\begin{proof}
		The equalities follow from \cref{hom-jacobi} applied to the triples $e_{xy}$, $e_{yz}$, $e_z$ and $e_x$, $e_{xy}$, $e_{yz}$, respectively.
	\end{proof}
	
	\begin{lemma}\label{not3non0comm}
		Let $\vf$ be a regular Hom-Lie structure on $I(X,K)$. If $e_{xy},e_{zw},e_{uv}\in \B$, then at least one of the commutators $[e_{xy},e_{zw}]$, $[e_{zw},e_{uv}]$ and $[e_{uv},e_{xy}]$ is zero.
	\end{lemma}
	\begin{proof}
		Suppose that $[e_{xy},e_{zw}]\neq 0\neq [e_{zw},e_{uv}]$. Then $[e_{xy},e_{zw}]\neq 0$ implies either $x=w$ and $y\neq z$ or $y=z$ and $x\neq w$. Similarly,   $[e_{zw},e_{uv}]\neq 0$ implies either $z=v$ and $w\neq u$ or $w=u$ and $z\neq v$.  So we have to analyze four cases.
		
		If $x=w=u$ and $y\neq z\neq v$ or if $y=z=v$ and $x\neq w\neq u$, we obtain $e_{uv}=e_{xv}$ or $e_{uv}=e_{uy}$ and, therefore, $[e_{xy},e_{uv}]=0$.
		
		If $x=w\neq u$ and $y\neq z=v$, then $u<v=z<w=x<y$. Consequently $u\neq y$ and $v\neq x$ which implies $[e_{xy},e_{uv}]=0.$
		
		The last possibility is $x\neq w=u$ and $y=z\neq v$. In this case $x<y=z<w=u<v$ and again $u\neq y$ and $v\neq x$, also implying $[e_{xy},e_{uv}]=0$.
	\end{proof}
	
	\begin{proposition}\label{0=identity}
		Let $\vf\in\laut(I(X,K))$ be a regular Hom-Lie structure on $I(X,K)$. Then
		\begin{enumerate}
			\item $\0_\vf=\id_\B$;\label{0_vf=id}
			\item $\sg_\vf(x,y)=1$ unless $x\in\Min(X)$ and $y\in\Max(X)$.\label{sg(xy)=1}
		\end{enumerate}
	\end{proposition}
	\begin{proof}
		Write, for simplicity, $\0=\0_\vf$ and $\sg=\sg_\vf$.
		
		\cref{0_vf=id}. Let $x<y$ and $u<v$ such that $\0(e_{xy})=e_{uv}$. Assume that $x\not\in\{u,v\}$. Then applying \cref{hom-jacobi} to the triple $e_u,e_{uv},e_x$ we get $[e_{uv},\vf(e_x)]=0$. Write $\vf(e_x)=\wtl\vf(e_x)+\rho_x$, where $\rho_x\in J(I(X,K))$. Then
		\begin{align*}
			[e_{uv},\wtl\vf(e_x)]=-[e_{uv},\rho_x].
		\end{align*}
		However, $[e_{uv},\wtl\vf(e_x)]$ is a scalar multiple of $e_{uv}$ because $\wtl\vf(e_x)\in D(X,K)$, while $[e_{uv},\rho_x]\in J^{k+1}$, where $k=l(u,v)$. Hence, $[e_{uv},\wtl\vf(e_x)]=0$. On the other hand,
		\begin{align*}
			[e_{uv},\wtl\vf(e_x)]=\sg(x,y)\m[\wtl\vf(e_{xy}),\wtl\vf(e_x)]=-\sg(x,y)\m\wtl\vf(e_{xy}),
		\end{align*}
		a contradiction. Similarly, assuming $y\not\in\{u,v\}$ and applying \cref{hom-jacobi} to the triple $e_u,e_{uv},e_y$ we get $[e_{uv},\wtl\vf(e_y)]=0$, which contradicts
		\begin{align*}
			[e_{uv},\wtl\vf(e_y)]=\sg(x,y)\m[\wtl\vf(e_{xy}),\wtl\vf(e_y)]=\sg(x,y)\m\wtl\vf(e_{xy}).
		\end{align*}
		Thus, $e_{xy}=e_{uv}$.
		
		\cref{sg(xy)=1}. Assume that $x\not\in\Min(X)$, so there exists $u<x$. Using \cref{basic-formulas} we have 
		\begin{align}\label{[e_ux_vf(e_xy)]=[vf(e_u)_e_uy]}
			[e_{ux},\vf(e_{xy})]=[\vf(e_u),e_{uy}].  
		\end{align}
		Since $\0=\id_\B$ by \cref{0_vf=id}, then $\vf(e_{xy})=\sg(x,y)e_{xy}+\rho_{xy}$, where $\rho_{xy}=\sum\rho_{xy}(a,b)e_{ab}$ with $a\le x<y\le b$ and $l(x,y)<l(a,b)$. Write also $\vf(e_u)=\wtl\vf(e_u)+\rho_u$, where $\rho_u\in J(I(X,K))$. Then \cref{[e_ux_vf(e_xy)]=[vf(e_u)_e_uy]} gives
		\begin{align}\label{[e_ux_sg(x_y)e_xy]-[wtl-vf(e_u)_e_uy]=-[e_ux_rho_xy]+[rho_u_e_uy]}
			[e_{ux},\sg(x,y)e_{xy}]-[\wtl\vf(e_u),e_{uy}]=-[e_{ux},\rho_{xy}]+[\rho_u,e_{uy}].
		\end{align}
		Now, 
		\begin{align*}
			[e_{ux},\sg(x,y)e_{xy}]&=\sg(x,y)[e_{ux},e_{xy}]=\sg(x,y)e_{uy},\\
			[\wtl\vf(e_u),e_{uy}]&=[\wtl\vf(e_u),\sg(u,y)\m\wtl\vf(e_{uy})]=\sg(u,y)\m\wtl\vf([e_u,e_{uy}])\\
			&=\sg(u,y)\m\wtl\vf(e_{uy})=\sg(u,y)\m \sg(u,y)e_{uy}=e_{uy}.
		\end{align*}
		Hence, the left-hand side of \cref{[e_ux_sg(x_y)e_xy]-[wtl-vf(e_u)_e_uy]=-[e_ux_rho_xy]+[rho_u_e_uy]} equals $(\sg(x,y)-1)e_{uy}$. We are going to show that the right-hand side of \cref{[e_ux_sg(x_y)e_xy]-[wtl-vf(e_u)_e_uy]=-[e_ux_rho_xy]+[rho_u_e_uy]} is zero at $(u,y)$. Let $l(u,y)=k$. Then $[\rho_u,e_{uy}]\in J^{k+1}$, so it is zero at $(u,y)$. Regarding $[e_{ux},\rho_{xy}]$, we have $e_{ux}\rho_{xy}=\sum_{y<b}\rho_{xy}(x,b)e_{ub}$, which is also zero at $(u,y)$. It follows that $\sg(x,y)=1$. Similarly, if $y\not\in\Max(X)$, then there exists $v>y$, and $[\vf(e_{xy}),e_{yv}]=[e_{xv},\vf(e_v)]$ gives $\sg(x,y)=1$.
	\end{proof}
	
	\begin{proposition}\label{0=id-and=sg=1=>vf-Hom-Lie}
		Let $\vf\in\wtl\laut(I(X,K))$ satisfying \cref{0_vf=id,sg(xy)=1} of \cref{0=identity}. Then $\vf$ is a regular Hom-Lie structure on $I(X,K)$.
	\end{proposition}
	\begin{proof}
		Assume \cref{0_vf=id,sg(xy)=1} and let $\0=\0_\vf$ and $\sg=\sg_\vf$. We will first prove \cref{hom-jacobi} for $a=e_{xy}$, $b=e_{zw}$ and $c=e_{uv}$, where $x<y$, $z<w$ and $u<v$.
		
		\textit{Case 1.} Exactly one of the products $[e_{xy},e_{zw}]$, $[e_{zw},e_{uv}]$ and $[e_{uv},e_{xy}]$ is non-zero, say $[e_{xy},e_{zw}]\ne 0$ and $[e_{zw},e_{uv}]=[e_{uv},e_{xy}]=0$. We thus need to prove that $[[e_{xy},e_{zw}],\vf(e_{uv})]=0$. Since $\vf(e_{uv})=\sg(u,v)e_{uv}$, the result follows from the usual Jacobi identity for the triple $e_{xy}$, $e_{zw}$, $e_{uv}$.
		
		\textit{Case 2.} Exactly two of the products $[e_{xy},e_{zw}]$, $[e_{zw},e_{uv}]$ and $[e_{uv},e_{xy}]$ are non-zero, say $[e_{xy},e_{zw}]\ne 0$, $[e_{zw},e_{uv}]\ne 0$ and $[e_{uv},e_{xy}]=0$. Since $\vf(e_{xy})=\sg(x,y)e_{xy}$ and $\vf(e_{uv})=\sg(u,v)e_{uv}$, we need to prove that 
		\begin{align*}
			\sg(u,v)[[e_{xy},e_{zw}],e_{uv}]+\sg(x,y)[[e_{zw},e_{uv}],e_{xy}]=0.
		\end{align*}
		Observe that $[[e_{xy},e_{zw}],e_{uv}]+[[e_{zw},e_{uv}],e_{xy}]=0$ by the usual Jacobi identity. It is thus enough to prove that $\sg(u,v)=\sg(x,y)$. It follows from $[e_{xy},e_{zw}]\ne 0$ that either $y=z<w$, in which case $y\not\in\Max(X)$, or $z<w=x$, in which case $x\not\in\Min(X)$. Hence, $\sg(x,y)=1$. Similarly, $[e_{zw},e_{uv}]\ne 0$ yields $z<w=u$, in which case $u\not\in\Min(X)$, or $v=z<w$, in which case $v\not\in\Max(X)$. Hence, $\sg(u,v)=1$ as well.
		
		Now let us prove \cref{hom-jacobi} for $a=e_{xy}$, $b=e_{zw}$ and $c=e_{uv}$, where some of the elements $a,b,c$ belong to $D(X,K)$.
		
		\textit{Case 1.} Exactly one of the elements $a,b,c$ belongs to $D(X,K)$, say $x=y$, $z<w$ and $u<v$.
		
		\textit{Case 1.1.} $x\not\in\{z,w\}\cup\{u,v\}$. Then we need to prove that $[[e_{zw},e_{uv}],\vf(e_x)]=0$. If $[e_{zw},e_{uv}]=0$, then there is nothing to prove. If $w=u$, then 
		\begin{align*}
			[[e_{zw},e_{uv}],\vf(e_x)]=[e_{zv},\vf(e_x)]=\sg(z,v)\m\vf([e_{zv},e_x]),
		\end{align*}
		where $[e_{zv},e_x]=0$ because $x\not\in\{z,v\}$. And if $z=v$, then
		\begin{align*}
			[[e_{zw},e_{uv}],\vf(e_x)]=-[e_{uw},\vf(e_x)]=-\sg(u,w)\m\vf([e_{uw},e_x]),
		\end{align*}
		where $[e_{uw},e_x]=0$ because $x\not\in\{u,w\}$.
		
		\textit{Case 1.2.} $x\in\{z,w\}\setminus\{u,v\}$. Then we need to prove that
		\begin{align}\label{[[e_x_e_zw]_vf(e_uv)]+[[e_zw_e_uv]_vf(e_x)]=0}
			[[e_x,e_{zw}],\vf(e_{uv})]+[[e_{zw},e_{uv}],\vf(e_x)]=0.
		\end{align}
		
		\textit{Case 1.2.1.} $x=z\not\in\{u,v\}$. Then 
		\begin{align*}
			[[e_x,e_{zw}],\vf(e_{uv})]=[e_{xw},\vf(e_{uv})]=\sg(u,v)[e_{xw},e_{uv}].
		\end{align*}
		If $[e_{xw},e_{uv}]=0$, there is nothing to prove. Otherwise $[e_{xw},e_{uv}]=e_{xv}$ because $x\ne v$. Since $x<w=u$ in this case, we have $u\not\in\Min(X)$ and consequently $\sg(u,v)=1$. Thus, $[[e_x,e_{zw}],\vf(e_{uv})]=e_{xv}$. On the other hand,
		\begin{align*}
			[[e_{zw},e_{uv}],\vf(e_x)]&=[e_{xv},\vf(e_x)]=\sg(x,v)\m[\vf(e_{xv}),\vf(e_x)]=\sg(x,v)\m\vf([e_{xv},e_x])\\
			&=-\sg(x,v)\m\vf(e_{xv})=-\sg(x,v)\m\sg(x,v)e_{xv}=-e_{xv}.
		\end{align*}
		Thus, \cref{[[e_x_e_zw]_vf(e_uv)]+[[e_zw_e_uv]_vf(e_x)]=0} holds.
		
		\textit{Case 1.2.2.} $x=w\not\in\{u,v\}$. Then 
		\begin{align*}
			[[e_x,e_{zw}],\vf(e_{uv})]=-[e_{zx},\vf(e_{uv})]=-\sg(u,v)[e_{zx},e_{uv}].
		\end{align*}
		If $[e_{zx},e_{uv}]=0$, there is nothing to prove. Otherwise $[e_{zx},e_{uv}]=-e_{ux}$ because $x\ne u$. Since $v=z<x$ in this case, we have $v\not\in\Max(X)$ and consequently $\sg(u,v)=1$. Thus, $[[e_x,e_{zw}],\vf(e_{uv})]=e_{ux}$. On the other hand,
		\begin{align*}
			[[e_{zw},e_{uv}],\vf(e_x)]&=-[e_{ux},\vf(e_x)]=-\sg(u,x)\m[\vf(e_{ux}),\vf(e_x)]=-\sg(u,x)\m\vf([e_{ux},e_x])\\
			&=-\sg(u,x)\m\vf(e_{ux})=-\sg(u,x)\m\sg(u,x)e_{ux}=-e_{ux}.
		\end{align*}
		Thus, \cref{[[e_x_e_zw]_vf(e_uv)]+[[e_zw_e_uv]_vf(e_x)]=0} holds.
		
		\textit{Case 1.3.} $x\in\{u,v\}\setminus\{z,w\}$. Then we need to prove that
		\begin{align*}
			[[e_{zw},e_{uv}],\vf(e_x)]+[[e_{uv},e_x],\vf(e_{zw})]=0.
		\end{align*}
		This case is the same as Case 1.2 up to the interchange between $e_{uv}$ and $e_{zw}$.
		
		\textit{Case 1.4.} $x\in\{u,v\}\cap\{z,w\}$.
		
		\textit{Case 1.4.1.} $x=u=z$. Then we need to prove that
		\begin{align*}
			0&=[[e_x,e_{xw}],\vf(e_{xv})]+[[e_{xv},e_x],\vf(e_{xw})]=\sg(x,v)[e_{xw},e_{xv}]-\sg(x,w)[e_{xv},e_{xw}],
		\end{align*}
		which is clearly satisfied because $[e_{xw},e_{xv}]=[e_{xv},e_{xw}]=0$.
		
		\textit{Case 1.4.2.} $x=u=w$. Then we need to prove that
		\begin{align*}
			0&=[[e_x,e_{zx}],\vf(e_{xv})]+[[e_{zx},e_{xv}],\vf(e_x)]+[[e_{xv},e_x],\vf(e_{zx})]\\
			&=-[e_{zx},\vf(e_{xv})]+[e_{zv},\vf(e_x)]-[e_{xv},\vf(e_{zx})]\\
			&=[e_{zv},\vf(e_x)]+(\sg(z,x)-\sg(x,v))e_{zv}.
		\end{align*}
		Since $z<x<v$, we have $x\not\in\Min(X)\cup\Max(X)$, whence $\sg(z,x)=\sg(x,v)=1$. Moreover,
		\begin{align*}
			[e_{zv},\vf(e_x)]=\sg(z,v)\m[\vf(e_{zv}),\vf(e_x)]=\sg(z,v)\m\vf([e_{zv},e_x])=0,
		\end{align*}
		because $[e_{zv},e_x]=0$.
		
		\textit{Case 1.4.3.} $x=v=z$. Then we need to prove that
		\begin{align*}
			0&=[[e_x,e_{xw}],\vf(e_{ux})]+[[e_{xw},e_{ux}],\vf(e_x)]+[[e_{ux},e_x],\vf(e_{xw})]\\
			&=[e_{xw},\vf(e_{ux})]-[e_{uw},\vf(e_x)]+[e_{ux},\vf(e_{xw})]\\
			&=-[e_{uw},\vf(e_x)]+(\sg(x,w)-\sg(u,x))e_{uw}.
		\end{align*}
		So, this case is the same as Case 1.4.2.
		
		\textit{Case 1.4.4.} $x=v=w$. Then we need to prove that
		\begin{align*}
			0&=[[e_x,e_{zx}],\vf(e_{ux})]+[[e_{ux},e_x],\vf(e_{zx})]=-\sg(u,x)[e_{zx},e_{ux}]+\sg(z,x)[e_{ux},e_{zx}],
		\end{align*}
		which is clearly satisfied because $[e_{zx},e_{ux}]=[e_{ux},e_{zx}]=0$.
		
		\textit{Case 2.} Exactly two of the elements $a,b,c$ belong to $D(X,K)$, say $x=y$, $z=w$ and $u<v$. We assume that $x\ne z$, since the case $x=z$ is obvious.
		
		\textit{Case 2.1.} $\{x,z\}\cap\{u,v\}=\emptyset$. Then \cref{hom-jacobi} is trivial.
		
		\textit{Case 2.2.} $x=u$, $z\not\in\{u,v\}$. Then we need to prove that $0=[[e_{xv},e_x],\vf(e_z)]=-[e_{xv},\vf(e_z)]$, which is true because $[e_{xv},\vf(e_z)]=\sg(x,v)\m\vf([e_{xv},e_z])$.
		
		\textit{Case 2.3.} $x=v$, $z\not\in\{u,v\}$. Then we need to prove that $0=[[e_{ux},e_x],\vf(e_z)]=[e_{ux},\vf(e_z)]$, which is true because $[e_{ux},\vf(e_z)]=\sg(u,x)\m\vf([e_{ux},e_z])$.
		
		\textit{Case 2.4.} $z=u$, $x\not\in\{u,v\}$. Then we need to prove that $0=[[e_z,e_{zv}],\vf(e_x)]=[e_{zv},\vf(e_x)]$, which is Case 2.2.
		
		\textit{Case 2.5.} $z=v$, $x\not\in\{u,v\}$. Then we need to prove that $0=[[e_z,e_{uz}],\vf(e_x)]=-[e_{uz},\vf(e_x)]$, which is Case 2.3.
		
		\textit{Case 2.6.} $x=u$ and $z=v$. Then we need to prove that
		\begin{align*}
			0=[[e_z,e_{xz}],\vf(e_x)]+[[e_{xz},e_x],\vf(e_z)]=-[e_{xz},\vf(e_x)]-[e_{xz},\vf(e_z)].
		\end{align*}
		But
		\begin{align*}
			[e_{xz},\vf(e_x)]&=\sg(x,z)\m\vf([e_{xz},e_x])=-\sg(x,z)\m\vf(e_{xz})=-e_{xz},\\
			[e_{xz},\vf(e_z)]&=\sg(x,z)\m\vf([e_{xz},e_z])=\sg(x,z)\m\vf(e_{xz})=e_{xz}.
		\end{align*}
		
		\textit{Case 2.7.} $x=v$ and $z=u$. Then we need to prove that
		\begin{align*}
			0=[[e_z,e_{zx}],\vf(e_x)]+[[e_{zx},e_x],\vf(e_z)]=[e_{zx},\vf(e_x)]+[e_{zx},\vf(e_z)],
		\end{align*}
		which is Case 2.6.
		
		\textit{Case 3.} All the elements $a,b,c$ belong to $D(X,K)$. This case is trivial, since $D(X,K)$ is commutative.
	\end{proof}
	
	The following two corollaries are immediate consequences of \cref{0=identity,0=id-and=sg=1=>vf-Hom-Lie}.
	
	\begin{corollary}\label{vf-Hom_lie<=>0=id-and-sg=1}
		Let $\vf\in\wtl\laut(I(X,K))$. Then $\vf$ is a regular Hom-Lie structure on $I(X,K)$ if and only if $\vf$ satisfies \cref{0_vf=id,sg(xy)=1} of \cref{0=identity}. 
	\end{corollary}
	
	\begin{corollary}\label{vf-Hom-Lie=>wtl-vf-Hom-Lie}
		If $\vf\in\laut(I(X,K))$ is a regular Hom-Lie structure on $I(X,K)$, then so is $\wtl\vf$.
	\end{corollary}
	
	The converse of \cref{vf-Hom-Lie=>wtl-vf-Hom-Lie} does not always hold, as the next remark shows.
	\begin{remark}
		Let $l(X)>1$ and fix $x<y$ with $x\not\in\Min(X)$. Define $\vf=\xi_{\bt}$, where $\bt=\dl+e_{xy}$. Then $\vf\in\inn_1(I(X,K))$ and for all $u<x$ we have
		\begin{align*}
			&[[e_{y},e_{u}],\vf(e_{ux})]+[[e_{u},e_{ux}],\vf(e_{y})]+[[e_{ux},e_{y}],\vf(e_{u})]\\
			&\quad=[e_{ux},(\dl+e_{xy})e_{y}(\dl-e_{xy})]=[e_{ux},(e_{y}+e_{xy})(\dl-e_{xy})]\\
			&\quad=[e_{ux},e_{y}+e_{xy}]=e_{uy},
		\end{align*}
		so $\vf$ is not a Hom-Lie structure, but $\wtl\vf=\id$ is.
	\end{remark}
	
	\begin{proposition}\label{xi-Hom-Lie<=>rho-in-Z}
		Let $\bt=\dl+\rho$, where $\rho\in J(I(X,K))$. Then $\xi_\bt$ is a regular Hom-Lie structure on $I(X,K)$ if and only if $\rho\in Z(J(I(X,K)))$.
	\end{proposition}
	\begin{proof}
		\textit{The ``only if'' part.} Let $\xi_\bt$ be a regular Hom-Lie structure on $I(X,K)$. Take $x<y$ such that $\rho(x,y)\ne 0$ and assume that $x\not\in\Min(X)$. Then there exists $u<x$. Applying \cref{hom-jacobi} to the triple $e_u,e_{ux},e_y$ we get 
		\begin{align*}
			0&=[e_{ux},(\dl+\rho)e_y\bt\m]=e_{ux}(\dl+\rho)e_y\bt\m=\rho(x,y)e_{uy}\bt\m,
		\end{align*}
		whence $e_{uy}=0$, a contradiction. Consequently, $\rho$ is a linear combination of $e_{xy}$, where $x\in\Min(X)$. It follows that $\rho^2=0$, because $e_{xy}e_{uv}=0$ for all $x<y$ and $u<v$ with $x,u\in\Min(X)$. Therefore, $\bt\m=\dl-\rho$. Now if $\rho(x,y)\ne 0$ and $y\not\in\Max(X)$, then taking $v>y$ and applying \cref{hom-jacobi} to the triple $e_{yv},e_v,e_x$ we get 
		\begin{align*}
			0&=[e_{yv},\bt e_x(\dl-\rho)]=-\bt e_x(\dl-\rho)e_{yv}=\rho(x,y)\bt e_{xv},
		\end{align*}
		a contradiction. Thus, $\rho$ is a linear combination of $e_{xy}$, where $x\in\Min(X)$ and $y\in\Max(X)$, that is, $\rho\in Z(J(I(X,K)))$.
		
		\textit{The ``if'' part.} Assume that $\rho\in Z(J(I(X,K)))$. Then $\rho^2=0$, so $\bt\m=\dl-\rho$ and $\xi_\bt(e_{xy})=e_{xy}+\rho'_{xy}$ for all $x\le y$, where $\rho'_{xy}\in J(I(X,K))$. In fact, $\rho'_{xy}\in Z(J(I(X,K)))$, because $\rho\in Z(J(I(X,K)))$, which is an ideal of $I(X,K)$. Therefore,
		\begin{align*}
			[[f,g],\xi_\bt(e_{xy})]=[[f,g],e_{xy}+\rho'_{xy}]=[[f,g],e_{xy}]
		\end{align*}
		for all $x\le y$ and $f,g\in I(X,K)$, because $[f,g]\in J(I(X,K))$. Hence, the Hom-Jacobi identity for $\xi_\bt$ reduces to the usual Jacobi identity in $I(X,K)$.
	\end{proof}
	
	\begin{proposition}\label{vf-Hom-Lie<=>wtl-vf-and-xi-Hom-Lie}
		Let $\vf=\xi_\bt\circ\wtl\vf\in\laut(I(X,K))$, where $\xi_\bt\in\inn_1(I(X,K))$. Then $\vf$ is a regular Hom-Lie structure on $I(X,K)$ if and only if so are $\wtl\vf$ and $\xi_\bt$.
	\end{proposition}
	\begin{proof}
		\textit{The ``only if'' part.} Assume that $\vf$ is a regular Hom-Lie structure on $I(X,K)$. Then so is $\wtl\vf$ by \cref{vf-Hom-Lie=>wtl-vf-Hom-Lie}. Hence, $\wtl\vf$ satisfies \cref{0_vf=id,sg(xy)=1} of \cref{0=identity}. Since $\0_\vf=\id$, then $\wtl\vf$ is proper by \cite[Lemma 2.6]{FKS2}. Moreover,
		\begin{align}\label{wtl-vf(e_xy)-cases}
			\wtl\vf(e_{xy})=
			\begin{cases}
				e_{xy}, & x<y,\ x\not\in\Min(X)\text{ or } y\not\in\Max(X),\\
				\sg_\vf(x,y)e_{xy}, & x<y,\  x\in\Min(X)\text{ and } y\in\Max(X),\\
				e_x+\af_x\dl, & x=y\text{ (by \cite[the proof of Lemma 2.6]{FKS2})},
			\end{cases}
		\end{align}
		where $\{\af_x\}_{x\in X}\sst K$. Let $x<y$. If $x\not\in\Min(X)$ or $y\not\in\Max(X)$, then  $\vf(e_{xy})=\xi_\bt(e_{xy})$ by \cref{wtl-vf(e_xy)-cases}, and consequently
		\begin{align}\label{[[f_g]_vf(e_xy)]=[[f_g]_xi_bt(e_xy)]}
			[[f,g],\vf(e_{xy})]=[[f,g],\xi_\bt(e_{xy})] 
		\end{align}
		for all $f,g\in I(X,K)$. Otherwise, $\wtl\vf(e_{xy})=\sg_\vf(x,y)e_{xy}$  by \cref{wtl-vf(e_xy)-cases} and $\xi_\bt(e_{xy})=e_{xy}$, so $\vf(e_{xy})=\sg_\vf(x,y)e_{xy}$. Consequently,
		\begin{align*}
			[[f,g],\vf(e_{xy})]=\sg_\vf(x,y)[[f,g],e_{xy}]=0=[[f,g],e_{xy}]=[[f,g],\xi_\bt(e_{xy})] 
		\end{align*}
		for all $f,g\in I(X,K)$, where the second and the third equalities are due to the fact that $e_{xy}\in Z(J(I(X,K)))$ and $[f,g]\in J(I(X,K))$. Now let $x=y$. Then $\wtl\vf(e_{xy})=e_{xy}+\af_x\dl$, so $\vf(e_{xy})=\xi_\bt(e_{xy})+\af_x\dl$, and again we have \cref{[[f_g]_vf(e_xy)]=[[f_g]_xi_bt(e_xy)]}. Thus, the left-hand side of \cref{hom-jacobi} for $\vf$ is the same as for $\xi_\bt$, and $\xi_\bt$ is also a regular Hom-Lie structure on $I(X,K)$.
		
		\textit{The ``if'' part.} Assume that $\wtl\vf$ and $\xi_\bt$ are regular Hom-Lie structures on $I(X,K)$. Then $\wtl\vf$ is of the form \cref{wtl-vf(e_xy)-cases}. The proof of the ``if'' part shows that the left-hand sides of \cref{hom-jacobi} for $\vf$ and $\xi_\bt$ are the same. Since $\xi_\bt$ is a regular Hom-Lie structure on $I(X,K)$, then so is $\vf$.
	\end{proof}
	
	We gather \cref{vf-Hom-Lie<=>wtl-vf-and-xi-Hom-Lie,xi-Hom-Lie<=>rho-in-Z,0=identity,vf-Hom-Lie=>wtl-vf-Hom-Lie,0=id-and=sg=1=>vf-Hom-Lie} into the following theorem which is the main result of our work.
	\begin{theorem}\label{main-theorem}
		Let $\vf\in\laut(I(X,K))$. Then $\vf$ is a regular Hom-Lie structure on $I(X,K)$ if and only if $\vf=\xi_\bt\circ\tau_{\0,\sg,c}$, where 
		\begin{enumerate}
			\item $\0=\id_\B$;\label{0-identity}
			\item $\sg(x,y)=1$ unless $x\in\Min(X)$ and $y\in\Max(X)$;\label{sg=1-unles-x-min-and-y-max}
			\item $\bt_D=\dl$ and $\bt_J\in Z(J(I(X,K)))$.\label{bt=dl+rho-rho-in-Z}
		\end{enumerate}
		Moreover, such $\bt$, $\0$, $\sg$ and $c$ are unique.
	\end{theorem}
	
	\begin{corollary}\label{vf=xi_bt-circ-M_sg+nu}
		Let $\vf\in\laut(I(X,K))$. Then $\vf$ is a regular Hom-Lie structure on $I(X,K)$ if and only if $\vf=\xi_\bt\circ M_\sg+\nu$, where $\bt$ and $\sg$ satisfy \cref{bt=dl+rho-rho-in-Z,sg=1-unles-x-min-and-y-max} of \cref{main-theorem} and $\nu$ is a central-valued linear map annihilating $J(I(X,K))$. Moreover, such $\bt$, $\sg$ and $\nu$ are unique.
	\end{corollary}
	\begin{proof}
		By \cref{main-theorem}, $\vf$ is a regular Hom-Lie structure on $I(X,K)$ if and only if $\vf=\xi_\bt\circ\tau_{\0,\sg,c}$ for unique $\0,\sg$ and $\bt$ satisfying \cref{0-identity,bt=dl+rho-rho-in-Z,sg=1-unles-x-min-and-y-max}. In this case $\wtl\vf=\tau_{\0,\sg,c}$, whose action on the standard basis is given by \cref{wtl-vf(e_xy)-cases}. Define $\nu(e_{xy})=0$ for $x<y$ and $\nu(e_x)=\af_x\dl$. Then $\nu$ is a unique central-valued linear map annihilating $J(I(X,K))$ and such that $\wtl\vf=M_\sg+\nu$, where $\sg$ is the extension of $\sg_{\vf}$ to $X^2_\le$ by $\sg(x,x)=1$. It follows that $\vf=\xi_\bt\circ\wtl\vf=\xi_\bt\circ M_\sg+\nu$, and such a representation is unique.
	\end{proof}
	
	In the case $l(X)=1$ the conditions \cref{sg=1-unles-x-min-and-y-max} on $\sg$ and $\bt_J\in Z(J(I(X,K)))$ on $\bt$ become redundant, while the condition $\bt_D=\dl$ on $\bt$ can be dropped at the cost of the uniqueness of $\bt$.
	\begin{corollary}\label{dropping-bt_D=id}
		Let $l(X)=1$ and $\vf\in\laut(I(X,K))$. Then $\vf$ is a regular Hom-Lie structure on $I(X,K)$ if and only if $\vf=\xi_\bt\circ M_\sg+\nu$, where $\nu$ is a central-valued linear map annihilating $J(I(X,K))$. Moreover, such $\sg$ and $\nu$ are unique.
	\end{corollary}
	\begin{proof}
		Item \cref{sg=1-unles-x-min-and-y-max} of \cref{main-theorem} is trivially satisfied, because for all $x<y$ we have $x\in\Min(X)$ and $y\in\Max(X)$. It is also clear that $\bt_J\in Z(J(I(X,K)))$, since $Z(J(I(X,K)))=J(I(X,K))$. So we only need to prove that $\bt_D=\dl$ can be dropped from \cref{main-theorem}\cref{bt=dl+rho-rho-in-Z}. This will be done via a suitable modification of $\sg$ that does not change $\xi_\bt\circ M_\sg$. Namely, we are going to show that for any $\bt\in I(X,K)^*$ and $\sg:X^2_<\to K^*$ there are $\bt'\in I(X,K)^*$ and $\sg':X^2_<\to K^*$ such that $\bt'_D=\dl$ and $\xi_\bt\circ M_\sg=\xi_{\bt'}\circ M_{\sg'}$. Write $\bt_D=\ve$ and $\bt_J=\rho$. Then $\bt\m=\ve\m-\ve\m\rho\ve\m$. Observe that 
		\begin{align*}
			(\xi_\bt\circ M_\sg)(e_{xy})&=\bt\sg(x,y)e_{xy}\bt\m=\ve(x,x)\ve(y,y)\m\sg(x,y)e_{xy},\ x<y,\\
			(\xi_\bt\circ M_\sg)(e_x)&=\bt e_x\bt\m=(\ve+\rho)e_x(\ve\m-\ve\m\rho\ve\m)=e_x+\rho\ve\m e_x-e_x\rho\ve\m.
		\end{align*}
		Thus, it suffices to take $\sg'(x,y)=\ve(x,x)\ve(y,y)\m\sg(x,y)$ (observe that we may define $\sg'(x,y)$ as we wish, since any map $X^2_<\to K^*$ is compatible with $\0$) and $\bt'=\dl+\rho\ve\m$.
	\end{proof}
	
	\begin{corollary}\label{Hom-Lie-T_n(K)}
		Let $X$ be a finite chain and $\vf\in\laut(I(X,K))$. Then $\vf$ is a regular Hom-Lie structure on $I(X,K)$ if and only if $\vf=\xi_\bt+\nu$, where $\nu$ is a unique central-valued linear map annihilating $J(I(X,K))$ and $\bt$ is either a unique element of $I(X,K)^*$ satisfying \cref{bt=dl+rho-rho-in-Z} of \cref{main-theorem} (if $|X|>2$) or an arbitrary element of $I(X,K)^*$ (if $|X|=2$).
	\end{corollary}
	\begin{proof}
		Let $\vf=\xi_\bt\circ M_\sg+\nu$ as in \cref{vf=xi_bt-circ-M_sg+nu}. Write $X=\{x_1,\dots,x_n\}$, where $x_1<\dots<x_n$. Observe that $\sg(x_i,x_j)=1$ unless $i=1$ and $j=n$. If $n>2$, then $\sg(x_1,x_n)=\sg(x_1,x_2)\sg(x_2,x_n)=1$. Hence, $M_\sg=\id$. Otherwise, $M_\sg=\xi_\eta$, where $\eta(x_1,x_1)=\sg(x_1,x_2)$, $\eta(x_2,x_2)=1$ and $\eta(x_1,x_2)=0$. Therefore, $\xi_\bt\circ M_\sg=\xi_{\bt\eta}$. Since $l(X)=1$ in this case, $\bt$ is an arbitrary element of $I(X,K)^*$ by \cref{dropping-bt_D=id}. Hence, so is $\bt\eta$.
	\end{proof}

	\begin{corollary}\label{Hom-Lie-normal-subgroup}
		The regular Hom-Lie structures on $I(X,K)$ form a normal subgroup of $\laut(I(X,K))$.
	\end{corollary}
	\begin{proof}
		Let $\vf,\psi\in\laut(I(X,K))$ be regular Hom-Lie structures on $I(X,K)$. Write $\vf=\xi_\bt\circ M_{\sg_\vf}+\nu_\vf$ and $\psi=\xi_\gm\circ M_{\sg_\psi}+\nu_\psi$ as in \cref{vf=xi_bt-circ-M_sg+nu}. Then $\vf|_{J(I(X,K))}=(M_{\sg_\vf})|_{J(I(X,K))}$ and $\psi|_{J(I(X,K))}=(M_{\sg_\psi})|_{J(I(X,K))}$ because 
		\begin{align*}
			(\nu_\vf)|_{J(I(X,K))}&=(\nu_\psi)|_{J(I(X,K))}=0,\\
			(\xi_\bt)|_{J(I(X,K))}&=(\xi_\gm)|_{J(I(X,K))}=\id_{J(I(X,K))},
		\end{align*}
		so
		\begin{align}\label{(vf-circ-psi)|_J=M_sg_vf_sg_psi}
			(\vf\circ\psi)|_{J(I(X,K))}=(M_{\sg_\vf*\sg_\psi})|_{J(I(X,K))}.
		\end{align}
		Take $x\in X$ and observe that
		\begin{align}\label{vf-circ-psi(e_x)=(xi_bt-M_sg_vf-xi_gm)(e_x)+mu(e_x)}
			(\vf\circ\psi)(e_x)=\vf(\xi_\gm(e_x)+\nu_\psi(e_x))&=(\xi_\bt\circ M_{\sg_\vf}\circ\xi_\gm)(e_x)+\mu(e_x),
		\end{align}
		where $\mu:=\nu_\vf+\nu_\psi+\nu_\vf\circ\nu_\psi$ is a central-valued linear map that annihilates $J(I(X,K))$. Now,
		\begin{align*}
			\xi_\gm(e_x)=(\dl+\gm_J)e_x(\dl-\gm_J)=e_x+\gm_Je_x-e_x\gm_J.
		\end{align*}
		Hence,
		\begin{align*}
			(\xi_\bt\circ M_{\sg_\vf}\circ\xi_\gm)(e_x)&=(\xi_\bt\circ M_{\sg_\vf})(e_x+\gm_Je_x-e_x\gm_J)\\
			&=\xi_\bt(e_x+M_{\sg_\vf}(\gm_J)e_x-e_xM_{\sg_\vf}(\gm_J))\\
			&=e_x+\bt_Je_x-e_x\bt_J+M_{\sg_\vf}(\gm_J)e_x-e_xM_{\sg_\vf}(\gm_J).
		\end{align*}
		Thus, defining $\eta:=\dl+\bt_J+M_{\sg_\vf}(\gm_J)$, we have $\eta_D=\dl$, $\eta_J\in Z(J(I(X,K)))$ and $(\xi_\bt\circ M_{\sg_\vf}\circ\xi_\gm)(e_x)=\xi_\eta(e_x)$. Together with \cref{vf-circ-psi(e_x)=(xi_bt-M_sg_vf-xi_gm)(e_x)+mu(e_x)} this gives $(\vf\circ\psi)|_{D(X,K)}=(\xi_\eta\circ M_{\sg_\vf*\sg_\psi}+\mu)|_{D(X,K)}$. Since, moreover, $(\vf\circ\psi)|_{J(I(X,K))}=(\xi_\eta\circ M_{\sg_\vf*\sg_\psi}+\mu)|_{J(I(X,K))}$ by \cref{(vf-circ-psi)|_J=M_sg_vf_sg_psi}, we conclude that $\vf\circ\psi=\xi_\eta\circ M_{\sg_\vf*\sg_\psi}+\mu$, so $\vf\circ\psi$ is a regular Hom-Lie structure on $I(X,K)$ by \cref{vf=xi_bt-circ-M_sg+nu}.
		
		Let $\vf$ be a regular Hom-Lie structure on $I(X,K)$. Write $\vf=\xi_\bt\circ\wtl\vf$ and $\vf\m=\xi_\gm\circ\wtl{\vf\m}$, where $\bt_D=\gm_D=\dl$ and $\bt_J\in Z(J(I(X,K)))$. Since $\wtl{\vf\m}=(\wtl\vf)\m$ by \cite[Proposition 4.2]{FKS}, then $\0_{\vf\m}=(\0_\vf)\m$ and $\sg_{\vf\m}(x,y)=\sg_\vf(x,y)\m$ for all $x<y$. It follows from \cref{vf-Hom_lie<=>0=id-and-sg=1} that $\wtl{\vf\m}$ is also a regular Hom-Lie structure on $I(X,K)$. Now observe that $\vf(e_{xy})=\wtl\vf(e_{xy})$ for all $x<y$. Then for all $x<y$ we have
		\begin{align*}
			e_{xy}=(\vf\m\circ\vf)(e_{xy})=(\xi_\gm\circ\wtl{\vf\m}\circ\wtl\vf)(e_{xy})=\xi_\gm(e_{xy}).
		\end{align*}
		Hence, $\gm e_{xy}=e_{xy}\gm$ for all $x<y$. But $\gm_D=\dl$, so $\gm_J e_{xy}=e_{xy}\gm_J$ for all $x<y$. Thus, $\gm_J\in Z(J(I(X,K)))$, and $\xi_\gm$ is a regular Hom-Lie structure on $I(X,K)$ by \cref{xi-Hom-Lie<=>rho-in-Z}. Then so is $\vf\m$ by \cref{vf-Hom-Lie<=>wtl-vf-and-xi-Hom-Lie}.
		
		We have proved that regular Hom-Lie structures on $I(X,K)$ form a subgroup of $\laut(I(X,K))$. It is normal by \cref{psi-vf-psi-inv-Hom-Lie}. 
	\end{proof}
	
	\section*{Acknowledgements}
	
	The second author was partially supported by CMUP, member of LASI, which is financed by national funds through FCT --- Fundação para a Ciência e a Tecnologia, I.P., under the project with reference UIDB/00144/2020.
	\bibliography{bibl}{}
	\bibliographystyle{acm}
	
\end{document}